\documentclass{amsart}
\usepackage{all2021}
\usepackage{xypic}
\usepackage{comment}
\usepackage{xcolor}

\newtheorem{question}[thm]{Question}

\renewcommand{\Vec}{\mathbf{Vec}}
\newcommand{\Sh}{\mathrm{Sh}}
\newcommand{\cP}{\mathcal{P}}

\newcommand{\prk}{\textnormal{prk}}

\newcommand{\ul}[1]{\underline{#1}}
\newcommand{\ceil}[1]{\lceil #1 \rceil}
\newcommand{\floor}[1]{\lfloor #1 \rfloor}

\begin{document}

\title[Strength and partition rank]{Strength and partition rank\\ under limits and field extensions}

\author[Bik]{Arthur Bik}
\address{D.E.~Shaw \& Co., New York}
\email{arthur.bik@deshaw.com}

\author[Draisma]{Jan Draisma}
\address{Mathematical Institute, University of Bern, Sidlerstrasse 5, 3012 Bern, Switzerland; and Department of Mathematics and Computer Science, Eindhoven University of Technology, P.O. Box 513, 5600MB, Eindhoven, the Netherlands}
\email{jan.draisma@unibe.ch}
\thanks{JD was partially supported by Swiss National Science
Foundation project grant 200021-227864.}

\author[Lampert]{Amichai Lampert}
\address{Mathematics Department, University of Michigan, 530 Church Street, Ann Arbor, MI 48109-1043, USA}
\email{amichai@umich.edu}
\thanks{AL was supported by NSF grants DMS-2402041 and DMS-1926686.}

\author[Ziegler]{Tamar Ziegler}
\address{Einstein Institute of Mathematics,
Edmond J. Safra Campus, 
The Hebrew University of Jerusalem,
Givat Ram. Jerusalem, 9190401, Israel}
\email{tamar.ziegler@mail.huji.ac.il}

\begin{abstract}
The strength of a multivariate homogeneous polynomial is the minimal
number of terms in an expression as a sum of products of lower-degree
homogeneous polynomials. Partition rank is the analogue for multilinear
forms. Both ranks can drop under field extensions, and both can jump in
a limit. We show that, for fixed degree and under mild conditions on
the characteristic of the ground field, the strength is at
most a polynomial in the border strength. We also establish an analogous
result for partition rank. Our results control both the jump under limits and
the drop under field extensions.
\end{abstract}

\maketitle

\section{Introduction and main results}

Let $K$ be a field
and let $f \in K[x_1,\ldots,x_n]_d$ be a form (i.e., homogeneous
polynomial) of degree $d$. In \cite{Schmidt85}, Schmidt introduced the following measure of complexity for $f$:

\begin{de}
The {\em strength} of $f$, denoted $s_K(f)$, is the minimal number $r$
of terms in any expression
\[ f=\sum_{i=1}^r a_i b_i \]
where the $a_i$ and $b_i$ are forms in $K[x_1,\ldots,x_n]$ of positive degrees. If $d=1$ and $f$ is nonzero, then we define $s_K(f)=\infty$.  
\end{de}

\subsection{Strength and field extensions}

Strength, also known as Schmidt rank, was independently introduced by
Green-Tao in their work on distribution of polynomials over finite fields
\cite{Green09} and by Ananyan-Hochster in their proof of Stillman's
conjecture\cite{Hochster16}. One of the key results in the work of both
Schmidt and Ananyan-Hochster is that for algebraically closed $K$,
the strength of a form is closely related to the codimension of its
singular locus.

This relationship also plays a central role in the arithmetic
applications studied by Schmidt and Green-Tao, where, however, $K$ is
not algebraically closed. It is therefore
important to understand the behaviour of strength under field
extensions. If $L$ is a field extension of $K,$ then we may also
consider $f$ as a form with coefficients in $L.$ The inequalities
$s_{\ol{K}}(f) \le s_L(f)\le s_K(f)$ are immediate, where $\ol{K}$ is
an algebraic closure of $K$. The following simple example shows that
these inequalities may be strict.

\begin{ex} \label{ex:Quadric}
For $d=2,$ $s_K(f)$ is the minimal codimension
of a subspace of $K^n$ on which the quadratic form $f$ vanishes
identically. So for $f=x_1^2+\cdots+x_n^2$ we have $s_\RR(f)=n$ and
$s_\CC(f)=\ceil{\frac{n}{2}}$.
\end{ex}

Bik-Draisma-Snowden give an example \cite[Example 3.4]{Draisma24a}
showing that, in general, the drop in strength can be arbitrarily large.
However, for semi-perfect fields they show in \cite[Theorem 1.3]{Bik24}
that $s_K(f) \le \Gamma(s_{\ol{K}}(f),d)$ where $\Gamma$ is some
non-explicit function that may depend on $K$.

For certain fields satisfying $\cha(K)=0$ or $\cha(K)>d$, Lampert-Ziegler
proved quantitative bounds for $\Gamma$ \cite{Lampert22}.  Our first
theorem gives a quantitative bound for any field satisfying this
condition.

\begin{thm}\label{thm:rkbar}
Suppose that $s_{\ol{K}}(f) = r$ and that $\cha(K)=0$ or $\cha(K)>d$. Then $s_K(f) \ll_d r^{d-1}$ for infinite $K.$ If $K$ is finite then  $s_K(f) \ll_d r^{d-1}\log r$.
\end{thm}

Here, $\ll_d$ means that the left-hand side is bounded by a constant
times the right-hand side, where the constant depends only on $d$
(and in particular not on $K$). 

\subsection{De-bordering strength}

Theorem~\ref{thm:rkbar} is a consequence of a stronger result which we now describe.
In the following definition, we write $\cP_d$ for
the affine space over $\ZZ$ with coordinates labelled by monomials in
$x_1,\ldots,x_n$ of degree $d$. The $K$-points of this scheme are the
homogeneous forms of degree $d$ in $x_1,\ldots,x_n$ with coefficients
in $K$.

\begin{de}
The {\em border strength} $\ul{s}(f)$ of $f \in K[x_1,\ldots,x_n]_d$ is defined as the minimal $r$ for
which there exist $d_1,\ldots,d_r \in \{1,\ldots,d-1\}$ such that $f$
is a $K$-point in the Zariski closure of the image of the morphism
\begin{align*} &\mu_{d_1,\ldots,d_r}: \prod_{i=1}^r (\cP_{d_i} \times
\cP_{d-d_i}) \to \cP_d, \\
&((a_1,b_1),\ldots,(a_r,b_r)) \mapsto \sum_{i=1}^r a_i b_i. 
\end{align*}
The border strength of a nonzero linear form is defined as $\infty$.
\end{de}

Write $\ol{K}$ for an algebraic closure of $K$. 
The border strength $\ul{s}(f)$ is also the minimal $r$ such that $f$
lies in the Zariski closure in $\overline{K}[x_1,\ldots,x_n]_d$ of forms
of strength $\leq r$. The latter set is the union of the images of the
$\mu_{d_1,\ldots,d_r}$, hence constructible by Chevalley's theorem. Hence
for $K=\CC$, since constructible sets have the same closure in the Zariski
topology as in the Euclidean topology, $\ul{s}(f)$ is also the minimal
$r$ such that there exists a sequence of forms of strength $\leq r$
that converges to $f$ in the Euclidean topology. This explains the
term border strength.

\begin{re}
We do not know an explicit example of a form over $\CC$ whose border
strength is strictly lower than its strength. However, such forms do
exist for $d=4$ \cite[Theorem 3]{Ballico22} and likely also for $d >
4$. For $d\leq 3$ they do not exist, since there strength is the minimal
codimension of a linear space on which the form vanishes identically
\cite[Proposition 2.2]{Derksen17}, and vanishing on a large subspace is
a Zariski-closed condition on the form.
\end{re}

We have the following fundamental inequalities: 
\begin{equation}\label{eq:str-ineq}
    s_K(f)  \geq s_{\ol{K}}(f) \geq \ul{s}(f).
\end{equation}

The goal of this paper is to upper bound $s_K(f)$ as a function of
$\ul{s}(f)$, as follows. 

\begin{thm} \label{thm:strength}
Suppose that $\ul{s}(f) = r$ and $\cha(K)=0$ or $\cha(K)>d$. 
For infinite $K$ we have
$s_K(f) \ll_d r^{d-1}$ and for finite $K$ we have $s_K(f) \ll_d
r^{d-1} \log r$. 
\end{thm}

Note that Theorem \ref{thm:rkbar} follows immediately by the inequalities \eqref{eq:str-ineq}. Following the theoretical computer
science literature \cite{Dutta22}, we call a result like that
in Theorem~\ref{thm:strength} a {\em de-bordering result}.

\subsection{Expressing low-strength polynomials as polynomials in
derivatives}

As a by-product of the proof of Theorem~\ref{thm:strength} and its
generalisations below, we will establish the following result. Given
$f \in K[x_1,\ldots,x_n]_d$, let $D(f)$ be the subspace of
$\bigoplus_{e=1}^{d-1} K[x_1,\ldots,x_n]_e$ spanned by all polynomials
of the form 
\[ \frac{\partial}{\partial x_{i_1}} \cdots \frac{\partial}{\partial
x_{i_l}} f \]
with $l \in \{1,\ldots,d-1\}$ and $i_1,\ldots,i_l \in \{1,\ldots,n\}$. 

\begin{thm} \label{thm:Df}
Suppose $\ul{s}(f) = r$ and $\cha(K)=0$ or $\cha(K)>d$. Then, if $K$ is infinite, 
$f$ lies in a subalgebra generated by $\ll_d r^d$ elements of $D(f)$. For finite $K,$ $f$ lies in a subalgebra generated by $\ll_d r^d\log r $ elements of $D(f).$
\end{thm}

\subsection{Connections to the literature}

Bounds for the strength of a form and for the partition rank of a
tensor in terms of its strength / partition rank over the algebraic
closure have been found for cubics in \cite{Adiprasito21}, for quartics
in \cite{Kazhdan23}, and for general degree but over certain fields
in \cite{Lampert22,Baily24}. In \cite{Adiprasito21}, these results
were applied to finite fields in order to bound strength/partition
rank in terms of \emph{analytic rank}. A well-known conjecture in
higher order Fourier analysis is that, for fixed degree, strength is
bounded by a constant multiple of analytic rank; for the currently
strongest results along these lines see \cite{Cohen21,Moshkovitz22}.
Adiprasito-Kazhdan-Ziegler conjectured \cite{Adiprasito21} that
Theorem \ref{thm:rkbar} holds with linear bounds, and showed that
this implies linear bounds for strength in terms of analytic rank. In
\cite{Moshkovitz22} and \cite{Baily24}, quasi-linear bounds were obtained
for finite fields. See \cite{Chen24} for additional results related to
this conjecture. We can now formulate two natural stronger versions of
Theorems~\ref{thm:strength} and~\ref{thm:Df}.

\begin{question}\label{q:lin-bounds}
    Can the bounds in Theorem \ref{thm:strength} and/or Theorem \ref{thm:Df} be taken linear in $r$?
\end{question}

Another strand of research, motivated by questions in theoretical computer
science, concerns {\em de-bordering} complexity
measures. A recent result in this direction is that the Waring rank of
a degree-$d$ form over $\CC$ admits an upper bound that is exponential
in its border Waring rank but {\em linear} in $d$ \cite{Dutta22}. 

Finally, our work is closely related to Karam's result that says that
bounded partition rank can be recognised by looking at all subtensors
of some given size \cite{Karam22}. Indeed, a quantitative version of
that statement involved degree estimations for equations vanishing on
tensors of bounded partition rank similar to those we will use below;
see \cite{Draisma23a}.

\subsection{De-bordering collective strength}

Our techniques for proving Theorem~\ref{thm:strength} also apply to
tuples of polynomials, to tensors, and to tuples of tensors. In this
and the next few sections, we record these additional results. 
First, for tuples of degree-$d$ forms, the following notion is a natural 
analogue of strength.

\begin{de}
Let $\ul{f}=(f_1,\ldots,f_m) \in K[x_1,\ldots,x_n]_d^m$. The {\em
strength} $s_K(\ul{f})$ is the minimum of $s_K(c_1 f_1 + \cdots + c_m f_m)$
over all $(c_1,\ldots,c_m) \in K^m \setminus \{(0,\ldots,0)\}$.
\end{de}

Thus the strength is $0$ if and only if $f_1,\ldots,f_m$
are linearly dependent.

\begin{de}
The {\em border strength} $\ul{s}(\ul{f})$ is defined as the
minimal $r$ such that there exist $d_1,\ldots,d_r \in \{1,\ldots,d-1\}$
for which $\ul{f}$ is a $K$-point in the image closure of the following
morphism over $\ZZ$:
\begin{align} \label{eq:Mu}
\mu_{m,d_1,\ldots,d_r}:&\AA^{m \times m} \times (\cP_d)^{m-1} \times 
\prod_{i=1}^{r} (\cP_{d_i} \times
\cP_{d-d_i}) 
\to (\cP_d)^m   \\   
&(g,(h_1,\ldots,h_{m-1}),((a_1,b_1),\ldots,(a_r,b_r))) \notag \\
&\quad \quad \mapsto 
g \cdot
(h_1,\ldots,h_{m-1},\mu_{d_1,\ldots,d_r}((a_1,b_1),\ldots,(a_r,b_r)))
\notag.
\end{align}
Here the notation $g \cdot $ means that in the $i$-th position, we take
the linear combination of the following $m$ polynomials with coefficients
given by the $i$-th row of $g$.
\end{de}

On $K$-points, the morphism in \eqref{eq:Mu} does the following: we pick $m-1$
arbitrary degree-$d$ forms and a degree-$d$ form of strength $\leq r$ (which admits a
decomposition via $\mu_{d_1,\ldots,d_r})$, and take $m$ arbitrary linear
combinations of these $r$ forms. Again, $\ul{s}(\ul{f})$ is the minimal $r$
such that $\ul{f}$ lies in the Zariski closure in $\ol{K}[x_1,\ldots,x_n]_d^m$
of the set of $m$-tuples of collective strength $\leq r$.

\begin{thm} \label{thm:collectivestrength}
Suppose $\ul{s}(\ul{f}) =r$ and $\cha(K)=0$ or $\cha(K)>d$. Then $s_K(\ul{f}) \ll_d m^3r^{d-1}$ if
$K$ is infinite and $s_K(\ul{f}) \ll_d 
m^3r^{d-1} \cdot \log(r+m)$ if $K$ is finite. Consequently,
$s_K(\ul{f})$ is also bounded by a polynomial in $s_{\overline{K}}(\ul{f})$.
\end{thm}

\subsection{De-bordering partition rank}

The results that we will state now require no assumptions on $\cha(K)$.
Let $V_1,\ldots,V_d$ be finite-dimensional $K$-vector spaces and let $t
\in V_1 \otimes \cdots \otimes V_d$; unless stated otherwise, tensor
products are over $K$. For a subset $I \subseteq [d]:=\{1,\ldots,d\}$
we write $V_I:=\bigotimes_{i \in I} V_i$. We have a canonical isomorphism
$V_I \otimes V_{I^c} \to V_1 \otimes \cdots \otimes V_d$, where $I^c=[d]
\setminus I$, and we will use this isomorphism to identify the two
spaces.

\begin{de}
The {\em partition rank} $\prk_K(t)$ of $t$ is the minimal $r$ in any
expression
\[ t=\sum_{i=1}^r a_i \otimes b_i \]
where each  $a_i \in V_{I_i}$ and $b_i \in V_{I_i^c}$ for some
proper, nonempty subset $I_i\subset [d]$. The {\em border partition rank}
$\ul{\prk}(t)$ is defined in a similar fashion as the border strength.
\end{de}

Partition rank was introduced by Naslund in \cite{Naslund20} and has
found applications in higher Fourier analysis and additive combinatorics
\cite{Cohen21,Gowers22,Lovett19,Milicevic19,Moshkovitz22,Naslund20b}. As
we will soon see, it is closely related to strength.

\begin{re} \label{re:EasyBound}
The partition rank of any tensor $t \in V_1 \otimes \cdots \otimes V_d$
is at most \linebreak
$r:=\min_j \dim(V_j)$, because $t$ admits an expression as above
with $I_i=\{j\}$ for all $i$, where $j$ attains the minimum, and with 
$a_i$ running through a basis of $V_j$. 
\end{re}

For a field extension $L$ of $K$, write $V_L$ for the $L$-vector space
$L \otimes_K V$. We may think of $t \in V_1 \otimes \cdots \otimes V_d$
as an element of the tensor product $(V_1)_L \otimes_L \cdots \otimes_L
(V_d)_L$, and then we have $\prk_L(t) \leq \prk_K(t)$.  In particular, as
before, we have
\[ \prk_K(t) \geq \prk_L(t) \geq \prk_{\overline{K}}(t) \geq \ul{\prk}(t). \]
Our de-bordering result for partition rank is as follows. 

\begin{thm} \label{thm:prank}
Suppose that $\ul{\prk}(t) = r.$ Then $\prk_K(t) \ll_d r^{d-1}$ if $K$ is infinite
and $\prk_K(t) \leq \ll_d r^{d-1} \cdot \log(r)$ if $K$ is
finite. In particular, $\prk_K(t)$ is also bounded by a polynomial in
$\prk_{\overline{K}}(t)$.
\end{thm}

\subsection{De-bordering collective partition rank}

\begin{de}
Let $m$ be a positive integer. The {\em (collective) partition rank}, $\prk_K(\ul{t})$,
of $\ul{t}=(t_1,\ldots,t_m) \in (V_1 \otimes \cdots \otimes V_d)^m$ is the
minimum of \linebreak $\prk_K(c_1 t_1 + \cdots + c_m t_m)$ over all $(c_1,\ldots,c_m)
\in K^m \setminus \{(0,\ldots,0)\}$. The {\em border (collective)
partition rank} of $\ul{t}$ is defined similarly as before. 
\end{de}

\begin{thm} \label{thm:collectiveprank}
Suppose $\ul{\prk}(\ul{t}) = r.$ Then $\prk_K(\ul{t}) \ll_d m^3r^{d-1}$ if $K$ is infinite and 
 $\prk_K(\ul{t}) \ll_d m^3r^{d-1} \log(r+m)$ if $K$ is finite. In particular, $\prk_K(\ul{t})$ is also bounded by a polynomial in $m$ and $ \prk_{\ol{K}}(\ul{t})$.
\end{thm}

\subsection{Proof overview}

We will establish de-bordering results for partition rank
and strength by following the approach in \cite{Draisma18b}. Indeed,
there it was shown that if one has any tensor property that is preserved
under linear maps, and a nontrivial polynomial equation satisfied by the
tensors with that property, then one can derive an upper bound on the
partition rank over $K$ of the tensors with that property. In this paper,
the tensor property is ``having border partition rank $\leq r$''. We find
bounds for such a polynomial equation, both in terms of dimensions of
the underlying vector spaces (this suffices for the result over infinite
fields) and in terms of degree of the equation (this is needed for the
result over finite fields), and use these to upper bound the partition
rank over $K$.

We observe that Theorem~\ref{thm:collectiveprank} implies
Theorem~\ref{thm:prank} and Theorem~\ref{thm:collectivestrength}
implies Theorem~\ref{thm:strength}. Furthermore, we
will show that Theorem~\ref{thm:collectiveprank} implies
Theorem~\ref{thm:collectivestrength}. It then suffices to prove
Theorem~\ref{thm:collectiveprank}.

\section{Preliminaries}
\subsection{Comparing strength and partition rank}

Let $V$ be a finite-dimensional vector space over $K$ and write $S^d V$
for the $d$-th symmetric power of $V$. Given any basis $x_1,\ldots,x_n$ of
$V$, the space $S^d V$ is canonically isomorphic to $K[x_1,\ldots,x_n]_d$,
and since strength is basis-invariant, we may transport the notion of
strength to $S^d V$. We have linear maps determined by 
\begin{align*} 
&\pi=\pi_d: V^{\otimes d} \to S^d V, \quad 
v_1 \otimes \ldots \otimes v_d \mapsto v_1 \cdots v_d  \text{ and}\\
&\iota=\iota_d: S^d V \to V^{\otimes d}, \quad v_1 \cdots v_d \mapsto \sum_{\pi \in S_d}
v_{\pi(1)} \otimes \cdots \otimes v_{\pi(d)}.
\end{align*}
These satisfy $\pi(\iota(f))=d! f$.  The following well-known proposition
relates the partition rank of elements in $V^{\otimes d}$ with the
strength of elements in $S^d V$.

\begin{prop} \label{prop:Sym}
For $t \in V^{\otimes d}$ and $f \in S^d V$ we have
\[ s_K(\pi(t)) \leq \prk_K(t) \text{ and }\prk_K(\iota(f)) \leq
D\cdot s_K(f) \text{ where } D:=\binom{d}{\floor{\frac{d}{2}}}. \]
\end{prop}

\begin{proof}
Since $\pi$ and $\iota$ are linear and strength/partition rank are
subadditive, it suffices to prove these inequalities when $\prk_K(t)=1$
and $s_K(f)=1$, respectively. For $\emptyset \subsetneq I \subsetneq [d]$, $a \in V^{\otimes
I}$ and $b \in V^{\otimes I^c}$
we have $\pi_d(a \otimes b)=\pi_e(a) \pi_{d-e}(b)$ (where $\pi_e$ is
short-hand for the concatenation of an isomorphism $V^{\otimes I} \to V^{\otimes
e}$ induced by a bijection $I \to [e]$ and the map $\pi_e$; and
similarly for $\pi_{d-e}$). This yields the first inequality.

For $a \in S^{e} V$ and $b \in S^{d-e} V$ we have 
\[ \iota(ab)=\sum_{I \subseteq [d],|I|=e} \iota_I(a) \otimes
\iota_{I^c}(b) \]
where $\iota_I: S^e V \to V^{\otimes I}$ stands for the composition of $\iota_e$
and an isomorphism $V^{\otimes e} \to V^{\otimes I}$ coming from a bijection $[e]
\to I$. This shows the second inequality. 
\end{proof}

\begin{re} \label{re:Sym}
Proposition~\ref{prop:Sym} readily implies the corresponding
inequalities 
\[ \ul{s}(\pi(t)) \leq \ul{\prk}(t) \text{ and } \ul{\prk}(\iota(f))
\leq D\cdot \ul{s}(f) \]
for border strength and border partition rank, as well as variants of
these for collective (border) strength and collective (partition) rank. 

For example, assuming for simiplicity that $K$ is infinite so that we
can take closures of sets of $K$-points rather than $K$-points of
image closures of morphisms, the locus $B$ of $m$-tuples of
tensors with collective border partition rank $\leq r$ is the closure
in $(V^{\otimes d})^m$ of the projection of the set
\[ \{((t_1,\ldots,t_m),(c_1,\ldots,c_m)) \mid \prk(\sum_i c_i t_i) \leq r\}
\subseteq (V^{\otimes d})^m \times (K^{m} \setminus \{0\}). \]
Let $\ul{s}=(s_1,\ldots,s_m)$ be a tuple in $B$. Then
$(\pi(s_1),\ldots,\pi(s_m))$ lies in the closure in $(S^d V)^m$ of the projection 
of the set 
\[ \{((\pi(t_1),\ldots,\pi(t_m)),(c_1,\ldots,c_m)) \mid \prk(\sum_i
c_i t_i) \leq r\} \subseteq  (S^d V)^m \times (K^m \setminus \{0\}
) \]
and Proposition~\ref{prop:Sym} implies that the condition $\prk(\sum_i
c_i t_i) \leq r\}$ implies the condition $s(\sum_i c_i \pi(t_i)) \leq r$.
Hence $(\pi(s_1),\ldots,\pi(s_k))$ lies in the closure of the set of
tuples of forms with collective strength $\leq r$.
\end{re}

\begin{prop}
Fix $d$ and $m$. Then Theorem~\ref{thm:collectiveprank} implies
Theorem~\ref{thm:collectivestrength}.
\end{prop}

\begin{proof}
Assume Theorem~\ref{thm:collectiveprank} and consider
\[ \ul{f}=(f_1,\ldots,f_m) \in K[x_1,\ldots,x_n]_d^m \cong (S^d V)^m. \]
Set $r:=\ul{s}(\ul{f})$, so $f$ lies in the Zariski closure of the set of
$m$-tuples of degree-$d$ forms over $\ol{K}$ of collective strength $\leq
r$. Applying $\iota$ over $\ol{K}$ and using Remark~\ref{re:Sym},
we find that $\ul{\prk}(\ul{t}) \leq D r$. 

Under the assumptions of Theorem~\ref{thm:collectivestrength}, $d!$ is
invertible in $K$, so that 
\[ s_K(\ul{f})=s_K(d!\ul{f})=s_K(\pi(\iota(f_1)),\ldots,\pi(\iota(f_m)))
\leq \prk_K(\ul{t}) \leq Q_m(Dr), \]
where the first inequality uses Remark~\ref{re:Sym} and the second
inequality uses Theorem~\ref{thm:collectiveprank}.
\end{proof}

\subsection{Behaviour under finite field extensions}

We record a simple and well-known relation between the (collective)
partition rank over $K$ and that over a field extension of $K$.

\begin{prop} \label{prop:FieldExtension}
Let $\ul{t} \in (V_1 \otimes \cdots \otimes V_d)^m$ and let $L$ be an extension
of $K$ with $e:=[L:K]<\infty$. Then the collective partition rank
satisfies $\prk_K(\ul{t}) \leq e \cdot \prk_L(\ul{t})$.
\end{prop}

\begin{proof}
Write 
\[ \sum_{k=1}^m c_k t_k=\sum_{i=1}^r a_i \otimes_L b_i \]
where $r:=\prk_L(\ul{t})$, $a_i \in L \otimes_K V_{I_i}$, $b_i \in L
\otimes_K V_{I_i^c}$, and $c_k \in L$, where the the $c_k$ are not
all zero. Assume, without loss of generality, that $c_m=1$. 
Let $z_1,\ldots,z_e$ be a basis of $L$ over
$K$. Then we can write $a_i=\sum_j z_j a_{ij}$ for certain tensors $a_{ij}
\in V_{I_i}$. We then have
\[ t_m + \sum_{k=1}^{m-1} c_k t_k=\sum_{i=1}^r \sum_{j=1}^e a_{ij} \otimes_L (z_j b_i). \]
The $c_k$ and $z_j b_i$ are a solution over $L$ to a certain system
of linear equations with coefficients from $K$. This system then also has a solution
over $K$, and such a solution witnesses $\prk_K(\ul{t}) \leq e \cdot r$. 
\end{proof}

\begin{re}
The corresponding statement also holds for strength, with the same
proof. The implication of our main theorems that the (collective)
strength or partition rank over $K$ can be bounded by the
corresponding quantity over $\ol{K}$ is therefore interesting only
for fields $K$ for which $\ol{K}$ is not finite-dimensional over $K$,
i.e., by the Artin-Schreier theorem, for fields that are not
real closed or algebraically closed.
\end{re}

\subsection{Subvarieties of tensor products}

It is convenient, for the time being, to assume that $K$ is infinite.
We will later explain how the reasoning can be adjusted to finite fields.
Let $\Vec$ be the category of finite-dimensional $K$-vector spaces, 
let $d$ be a natural number, and denote by $\Vec^d$ the category whose
objects and morphisms are $d$-tuples of objects and morphisms in
$\Vec$, respectively. 

\begin{de}
A functor $P:\Vec^d \to \Vec$ is {\em polynomial of degree $\leq e$} if
for any $V,W \in \Vec^d$ the map $P:\Hom(V,W) \mapsto \Hom(P(V),P(W))$
is polynomial of degree $\leq e$.
\end{de}

By a polynomial functor $\Vec^d \to \Vec$ we will always mean one that
is polynomial of degree at most some $e$. In fact, the most interesting
ones for this paper are the functors
\[ T_I:\Vec^d \to \Vec, (V_1,\ldots,V_d) \mapsto \bigotimes_{i \in I}
V_i=V_I \]
where $I$ is a subset of $[d]$.

We will think of $P(V)$ not just as a $K$-vector space, but also as the
spectrum of the symmetric algebra of $P(V)^*$, an affine scheme of finite
type over $K$. For $\phi \in \Hom_{\Vec^d}(V,W)$ the linear map $P(\phi)$
induces a morphism $P(V) \to P(W)$ of affine schemes over $K$.

\begin{de}
A {\em closed subvariety} of a polynomial functor $P:\Vec^d \to \Vec$ is
a rule $X$ that assigns to every tuple $V \in \Vec^d$ a reduced, closed
subscheme $X(V) \subseteq P(V)$ in such a manner that for any $\phi:V
\to W$ the map $P(\phi)$ maps $X(V)$ into $X(W)$. We then write $X(\phi)$
for the restriction $X(V) \to X(W)$ of the morphism $P(\phi)$ to $X(V)$.

Let $X,Y$ be closed subvarieties of polynomial functors $\Vec^d \to \Vec$. A
{\em morphism} $\alpha:X \to Y$ is the data of a morphism $\alpha_V:X(V)
\to Y(V)$ of affine algebraic varieties over $K$ for all $V \in \Vec^d$
in such a manner that for all $V,W \in \Vec^d$ and all $\phi \in
\Hom_{\Vec^d}(V,W)$ the following diagram commutes:\\
\[
\xymatrix{
X(V) \ar[r]^{\alpha_V} \ar[d]_{X(\phi)} & Y(V) \ar[d]^{Y(\phi)} \\
X(W) \ar[r]_{\alpha_W} & Y(W). 
}
\]
\end{de}

The most important example of a closed subvariety for us lives in
$T_{[d]}^m$ ($m$-tuples of tensors) and is defined by
\[ X_r(V) := \{\ul{t} \in T_{[d]}^m(V) \mid \ul{\prk}(\ul{t}) \leq r\}; \]
the variety of $m$-tuples of tensors of border partition rank $\leq r$. 

\begin{de}
Given a tuple $U \in \Vec^d$ and a polynomial functor $P:\Vec^d \to
\Vec$, we get a new polynomial functor $\Sh_U P:\Vec^d \to \Vec$, the
{\em shift of $P$ over $U$}, by setting
\[ (\Sh_U P)(V):=P(U \oplus V)
\text{ and } (\Sh_u P)(\phi):=P(\id_U \oplus \phi). \]
Furthermore, if $X$ is a closed subvariety of $P$, then we define $\Sh_U X$
as $(\Sh_U X)(V):=X(U \oplus V) \subseteq (\Sh_U P)(V)$; this is a closed
subvariety of $\Sh_U P$.
\end{de}

The following important example of a shift will be used intensively.

\begin{ex}
Since tensor products distribute over direct sums, we have
\[ (\Sh_U T_{[d]})(V)  = \bigoplus_{I \subseteq [d]} U_{I^c} \otimes
V_I, \]
so $\Sh_U T_{[d]}$ is a direct sum, over all subsets $I$ of $[d]$,
of $\prod_{i \in I^c} (\dim(U_i))$ copies of $T_I$. Note that there is
precisely one summand of degree $d$, namely, that with $I=[d]$; this
summand 
is $U_\emptyset \otimes V_{[d]}$, which we identify with
$V_{[d]}=T_{[d]}(V)$. The
quotient $(\Sh_U T_{[d]})/T_{[d]}$ is a polynomial functor of degree
$d-1$.

Similarly, for any subset $I \subseteq [d]$, $\Sh_U T_I^m$
equals $T_I^m$ plus a direct sum of copies of $T_{I'}^{m'}$ 
where $I' \subseteq I$ and where $m'<m$ if $I'=I$. 
\end{ex}

\begin{de}
Given a closed subvariety $X$ of a polynomial functor $P$ and a function
$h \in K[X(0)]$, we may think of $h$ as a function on every $X(V)$ via
the composition with the map $X(0_{V \to 0}):X(V) \to X(0)$. We define
the functor $X[1/h]$ from $\Vec^d$ to affine $K$-schemes by sending
$V$ to the basic open subset of $X(V)$ defined by the nonvanishing
of $h$.  This is a closed subset of the polynomial functor $V \mapsto
K \oplus P(V)$ defined by the equations for $X(V) \subseteq P(V)$ and
the equation $y \cdot h = 1$, where $y$ is the coordinate on the affine
line corresponding to the summand $K$. 
\end{de}

\section{Proof of Theorem~\ref{thm:collectiveprank} for infinite $K$}
\label{sec:Infinite}

We continue to assume that $K$ is infinite. We have 
\[ T_{[d]}^m (K^{n_1},\ldots,K^{n_d})=(K^{n_1} \otimes \cdots
\otimes K^{n_d})^m. \]
In this space we have the standard basis vectors 
\[ e_{(i_1,\ldots,i_d),\ell}
:=(0,\ldots,0,e_{i_1} \otimes \cdots \otimes e_{i_d},0,\ldots), \]
where the nonzero entry is in position $\ell \in [m]$ and where
the $e_{i_j}$ are the standard basis vectors in $K^{n_j}$. We write
$x_{(i_1,\ldots,i_d),\ell}$ for the corresponding standard coordinates.

The following theorem is a summary of results proved in \cite[\S
4]{Draisma18b}, but generalised from tensors to $m$-tuples of tensors.
Informally, the theorem says that on an open subset where a partial
derivative of a defining function for a closed subset $X \subseteq
T^m_{[d]}$ is nonzero, the coordinates of the $m$-th tensor can be
reconstructed from those of the remaining tensors. 

\begin{thm} \label{thm:Master}
Let $X \subseteq T_{[d]}^m$ be a closed subvariety and assume that $X$
is defined over the prime field of $K$ in the following sense: for all
$n_1,\ldots,n_d$, $X(K^{n_1},\ldots,K^{n_d}) \subseteq (K^{n_1} \otimes
\cdots \otimes K^{n_d})^m$ is the zero set of polynomials in the standard
coordinates with coefficients from the prime field $F$ of $K$.

Let $U=(U_1,\ldots,U_d) \in \Vec^d$ with $U_j=K^{n_j}$ be such that
$X(U) \subsetneq (T_{[d]}(U))^m$. Let $f$ be a nonzero polynomial with
coefficients in $F$ that vanishes identically on $X(U)$.  Let $x_1$
be the standard coordinate $x_{(1,\ldots,1),m}$ and set
\[ h:=\frac{\partial f}{\partial x_1}. \]
Let $M$ be the $m$-th copy of $T_{[d]}$ in $T_{[d]}^m$. It is then
naturally a summand of $\Sh_U T_{[d]}^m$, and the projection $\Sh_U
T_{[d]}^m \to (\Sh_U T_{[d]}^m)/M$ restricts to a closed embedding,
defined over $F$,
\[ \psi: (\Sh_U X)[1/h] \to ((\Sh_U T_{[d]}^m)/M)[1/h]. \]
Furthermore, there exists a morphism, defined over $F$,
\[ \sigma: ((\Sh_U T_{[d]}^m)/M)[1/h] \to (\Sh_U T_{[d]}^m)[1/h] \]
such that $\sigma \circ \psi$ is the identity on $(\Sh_U X)[1/h]$. 
\end{thm}

\begin{re}
For $m=1$ the fact that this projection restricts to a closed embedding
is \cite[\S 4.6]{Draisma18b} together with \cite[Lemma 33]{Draisma18b},
and this latter lemma also gives the map $\sigma$, which is called $\Psi$
there. The generalisation to $m$ copies is straightforward.

A difference with \cite{Draisma18b} is that there it is
required that $f$ is homogeneous of minimal degree, but neither of those
properties is necessary for the statement above.  However, if one does
not impose that $f$ is of minimal degree, then it may be that $(\Sh_U
X)[1/h]$ is empty and the theorem holds trivially.  This is no problem,
because in the application below, we then replace $X$ by its subvariety
defined by $h$, which is of lower degree than $f$.
\end{re}

\begin{cor} \label{cor:Master}
In the setting of Theorem~\ref{thm:Master}, for any $V \in \Vec^d$ and
any $K$-point $\ul{t}$ of $X(V)$ for which there exists
an $m$-tuple $\phi:V \to U$ of 
$K$-linear maps with 
$h(T_{[d]}^m(\phi)\ul{t}) \neq 0$, the collective partition
rank $\prk_K(\ul{t})$ of $\ul{t}$ is at most 
\begin{equation}  \label{eq:HarderBound}
m\cdot \sum_{I \subseteq [d]: 1 \leq |I| \leq \floor{d/2}}
\prod_{j \in I^c} n_j  
\end{equation}
\end{cor}

\begin{proof}
Let $\tilde{\ul{t}}$ be the image of $\ul{t}$ in $T_{[d]}^m(U \oplus V)$ under
the linear map $T_{[d]}^m(\phi \oplus \id_V)$. So the $i$-th component
of $\tilde{\ul{t}}$ equals $t_i \in V_{[d]}$ plus a remainder in 
\[ \bigoplus_{I \subsetneq [d]} U_{I^c} \otimes V_I
= \sum_{i=1}^d U_i \otimes \bigotimes_{j \neq i} (U_j \oplus V_j). \]
We analyse $t_m$, which is the part of $\tilde{t}_m$ in $M$.

We have 
\[ h(\tilde{t})=h(T^m_{[d]}(\phi)\ul{t})\neq 0, \]
so that we may regard
$\tilde{\ul{t}}$ as a $K$-point of $(\Sh_U X)(V)[1/h]$. Therefore, 
Theorem~\ref{thm:Master} yields the equality
$\tilde{t}=\sigma(\psi(\tilde{t}))$. The $M$-component
of $\sigma$ is a morphism $((\Sh_U T_{[d]}^m)/M)[1/h] \to M=T_{[d]}$ and
hence, by basic representation theory for $\prod_i \GL(V_i)$, a linear
combination (with coefficients in $K[T_{[d]}^m(U)][1/h]$) of morphisms
that take the components in some summands $T_{I_1},\ldots,T_{I_k}$
for which $I_1 \sqcup \cdots \sqcup I_k=[d]$ and returns their tensor
product. (See the paragraph \cite[\S 4.8]{Draisma18b} on covariants
for details.)

Among these morphisms, those with $k=1$ are isomorphisms from the
$m-1$ copies of $T_{[d]}$ in $\Sh_U T_{[d]}^m/M$; this is where
$t_1,\ldots,t_{m-1}$ live. So for suitable $c_1,\ldots,c_{m-1} \in K$,
$t':=t_m+\sum_{i=1}^{m-1} c_i t_i$ is a linear combination of
tensors in the images of those morphisms with $k \geq 2$. For these, the smallest
$I_j$ has cardinality at most 
$\floor{\frac{d}{2}}$. So we find that $t'$
has partition rank at most 
the number of (copies of) $T_I$
with $|I| \leq \floor{\frac{d}{2}}$ in $\Sh_U T_{[d]}^m$, which is
\[ m\cdot\sum_{I \subseteq [d]: 1 \leq |I| \leq \floor{d/2}} \prod_{j \in I^c} n_j.
\]
Thus the collective partition rank of $\ul{t}$ is bounded by
the number above.
\end{proof}

We record the following corollary of the proof.
\begin{de}
A {\em proper contraction}
of $\ul{t} \in T_{[d]}(V)^m$ is any tensor of the form
\[ \sum_{i=1}^m \xi_i(t_i) \in T_{I}(V), \]
where $I$ is a nonempty proper subset of $[d]$, and $\xi_i \in
T_{I^c}(V^*)$ for all $i$.
\end{de}

\begin{cor} \label{cor:Algebra}
There exists an $N$ such
that for any $\ul{t}$ as in Corollary~\ref{cor:Master} there is a $K$-linear
combination $t' = t_m+\sum_{i=1}^{m-1} c_i t_i$ which is contained in a subalgebra of \linebreak 
$T(V) = \oplus_{e=0}^\infty T_{[d]}(V)$ generated by  
$ \le \sum_{j=1}^d n_j + m\cdot \prod_{j\in [d]} (n_j+1)$ proper
contractions of $\ul{t}$. 
\end{cor}

Such a tensor product is of the form $s_1 \otimes \cdots \otimes s_e$
where $e \geq 2$, $s_i \in T_{I_i}(V)$ and $I_1,\ldots,I_e$ form a
partition of $[d]$ into nonempty sets. 

\begin{proof}
Without loss of generality, $\ul{t}$ is {\em concise}: each
$V_j$ is the smallest space $V_j' \subseteq V_j$ with $\ul{t} \in
T_{[d]}^m(V_1,\ldots,V_{j-1},V_j',V_{j+1},\ldots,V_d)$. Then each element
of $V_j$ is a proper contraction of $\ul{t}$.

Let $V_j'$ be the kernel of $\phi_j$ and let $U_j'$ be a complement of
$V_j'$ in $V_j$. Then the restriction of $\phi$ to $U'$ is injective,
and we may regard this as an inclusion, so that $U_j' \subseteq U_j$ for
all $j$ and $\ul{t}$ lies in $T^m_{[d]} (U' \oplus V') \subseteq T^m_{[d]}
(U \oplus V')$. We will treat $\ul{t}$ as we treated $\tilde{\ul{t}}$ in the
proof of Corollary~\ref{cor:Master}. Consider the linear combination
$t':=t_m+\sum_{i=1}^{m-1} c_i t_i$ constructed
in Corollary~\ref{cor:Master}. We want to show that this admits a
decomposition as desired.

Choose a basis of each $U_j$ containing a basis of $U_j'$. 
Each $t_i$ admits a unique decomposition
\[ t_i=\sum_u u \otimes t_{i,u} \]
where $u$ runs over all $\prod_{j=1}^d (1+ n_j)$ ``monomials'' in the
chosen bases of the $U_j$: each lives in a tensor product $\bigotimes_{j
\in I} U_j$ for some subset $I \subseteq [d]$, while its coefficient $t_{i,u}$ lives in the tensor product
$\bigotimes_{j \in I^c} V_j'$. From the fact that $\ul{t}$ lies in $T^m_{[d]}(U'
\oplus V')$ it follows that that $t_{i,u}$ is zero whenever $u$ involves
basis elements of some $U_j$ that are not in $U_j'$. So in all nonzero
terms above, $u$ is contained in the subalgebra generated by the union
of the bases of $U_j',$ which consists of  $\le \sum_{j=1}^{d} n_j$
proper contractions of $\ul{t}$. 

Now consider a monomial $u \in \bigotimes_{j \in I} U_j'$ of positive
degree. Then inclusion-exclusion yields:
\[ t_{i,u}=\sum_{u',u'': u'=u \otimes u''} 
(-1)^{|u'|-|u|} \xi_{u'}(t_i) \otimes u'' \]
where the sum is over all monomials $u'$ that are multiples of $u$,
hence $u' \in \bigotimes_{j \in I'} U_j'$ with $I' \supseteq I$, 
and $\xi_{u'}$ is a linear form on the space $\bigotimes_{j \in I'}
(U_j' \oplus V_j')$ that is $1$ on $u'$, zero on any other monomial in
$\bigotimes_{j \in I'} U_j'$, and zero on any tensor product of vectors
containing a vector in some $V_j'$ with $j \in I'$. The number of
contractions $\xi_{u'}(t_i)$ appearing in these expressions for any
$i$ and $u$ of positive degree is $\leq m\cdot \prod_{j\in [d]}
(n_j+1).$ 

Now, $t'$ has a unique decomposition like the $t_{i,u}$:
\[ t'=\sum_u u \otimes t'_u, \]
where each $t'_u$ is a linear combination of the $t_{i,u}$, and
hence is contained in the subalgebra we have so far described, as long as $u$ has positive degree.
 On the other hand, in the proof of Corollary~\ref{cor:Master}
we saw that $t'_u=t'_1$, where $u=1$ is the unique degree-$0$ monomial,
is a sum of tensor products of the $t_{i,u}$ for monomials $u$
of positive degree. This completes the proof.
\end{proof}

\begin{prop} \label{prop:Bound}
The closed subvariety $X_r(V):=\{\ul{t} \in T_{[d]}^m(V) \mid
\ul{\prk}(\ul{t}) \leq
r\}$ of $T_{[d]}^m$ satisfies the conditions of
Theorem~\ref{thm:Master} with $n_1=\ldots=n_d=n$ where $n$ is
minimal such that 
\begin{equation} \label{eq:nBound}  n^d > m^2 +
r(n^{d-1}+n). \end{equation}
\end{prop}

\begin{proof}
For any choice of $r$ proper nonempty subsets $I_1,\ldots,I_r$ of $[d]$,
each of size $\leq \floor{d/2}$, and $U:=(K^n,\ldots,K^n)$ consider the
parameterisation
\begin{align} 
K^{m \times m} \times T_{[d]}^{m-1}(U) \times \prod_{i=1}^r
\left( T_{I_i}(U) \times T_{I_i^c}(U) \right) \to T_{[d]}^m(U) , \label{eq:Parm}\\
(g,(t_1,\ldots,t_{m-1}),((a_i,b_i))_i) \mapsto 
g \cdot (t_1,\ldots,t_{m-1},\sum_i a_i \otimes b_i) \notag
\end{align}
which is the tensor analogue to the map in \eqref{eq:Mu}. The locus
of tensor tuples of collective partition rank at most $r$ is the union of the images
of these parameterisations over all $r$-tuples $(I_1,\ldots,I_r)$ as
above. Since the parameterisation is defined over the prime field, $X$
is defined over the prime field.

The dimension of $V_{I_i} \times V_{I_i^c}$ is at most $n^{d-1}+n$, with
equality if $I_i$ or $I_i^c$ is a singleton. Hence the left-hand side in the
parameterisation has dimension at most 
\[ (m-1)n^d + m^2 + r(n^{d-1}+n) < mn^d.  
\]
By inequality \eqref{eq:nBound},
the parameterisation is not dominant into $T_{[d]}^m(U)$, hence there
exists a nonzero polynomial $F$ that vanishes on the image. Taking the
product of these $F$s for the finitely many choices of $(I_1,\ldots,I_r)$
gives a nonzero polynomial vanishing on all of $X_r(U)$.
\end{proof}

We note the crucial fact that for $d$ fixed, the lower bound
on $n$ imposed by \eqref{eq:nBound} is {\em linear} in $r+m^{2/d}$. We can now
prove Theorem~\ref{thm:collectiveprank} in the case where $K$ is
infinite. 

\begin{proof}[Proof of Theorem~\ref{thm:collectiveprank} when $K$ is
infinite.]

By Proposition~\ref{prop:Bound}, $X_r(U) \subsetneq T^m_{[d]}(U)$ for
$U_i:=K^n, i=1,\ldots,d$ with $n$ minimal such that \eqref{eq:nBound}
holds. We claim that, for any $V \in \Vec^d$, the collective partition
rank $\prk_K(\ul{t})$ of any $K$-point $\ul{t}$ of $X_r(V)$ is at most
\begin{equation} \label{eq:PolBound} 
m\cdot \sum_{e=1}^{\lfloor \frac{d}{2} \rfloor} \binom{d}{e} n^{d-e}. 
\end{equation}
Since $n$ is linear in $r+m^{2/d}$, this expression is $\ll_d m^3 r^{d-1}$. 

We prove the claim with $X_r$ replaced by any $\Vec^d$-subvariety $Y$
of $X_r$ that is defined over the prime field $F$ of $K$. This has the
advantage that we can do induction on the minimal degree $\delta(Y)$
of a nonzero polynomial vanishing on $Y(U)$: when proving it for $Y$,
we may assume that the claim holds for all $\Vec^d$-subvarieties $Y'$
of $X_r$ for which $\delta(Y')<\delta(Y)$.

Thus let $f \in K[T^m_{[d]}(U)]$ be a nonzero polynomial, with
coefficients in $F$, of minimal degree that vanishes on $Y$. If $f$ is
constant, then $Y$ is empty and the claim holds trivially. Otherwise,
if $\cha(K)=p$, then since $f$ is of minimal degree, it is not a $p$-th
power of a polynomial. Hence in some monomial in $f$, some variable
has an exponent that is not divisible by $p$. 
Without loss of generality, this variable is the variable $x_1$ in
Theorem~\ref{thm:Master}. Construct construct $h$ as in that theorem. By
construction, $h$ is nonzero and has smaller degree than $f$.

Let $Y' \subseteq Y$ be the $\Vec^d$-subvariety defined by 
\begin{equation} \label{eq:Yprime} 
Y'(V):=\{\ul{t} \in Y(V) \mid \forall \phi \in \Hom(V,U):
h(T^m_{[d]}(\phi)\ul{t}) = 0\}. 
\end{equation}
This is again defined over $F$. Indeed, instead of letting $\phi$ run
over all of $\Hom(V,U)$, we can take for $\phi$ a tuple of matrices with
entries that are variables, expand $h(T^m_{[d]}(\phi)\ul{t})$ as a polynomial
in these variables, and take the coefficients; these define $Y'$---we
use here that $K$ is infinite.

On $Y'$ the nonzero polynomial
$h$ of strictly lower degree than $\deg(f)$ vanishes, so that
$\delta(Y')<\delta(Y)$. Hence the claim holds for $Y'$ by the induction
hypothesis. On the other hand, for $\ul{t}$ a $K$-point of $Y'(V) \setminus Y(V)$ we can
apply Corollary~\ref{cor:Master} to conclude that $\ul{t}$ has collective
partition rank at most \eqref{eq:PolBound}.
\end{proof}

\section{Proof of Theorem~\ref{thm:collectiveprank} for finite $K$}

In this section we assume that $\cha K=p>0$ and drop the requirement that
$K$ is infinite. We consider the affine scheme $X_r \subseteq T_{[d]}^m$
whose $\ol{K}$-points are the tuples of border collective partition rank
$\leq r$. Let $U_i=K^n$ where $n$ is such that \eqref{eq:nBound}
holds and let $f$ be a nonzero polynomial, with coefficients in the prime
field, of minimal degree in the ideal of $X_r(U)$. Assume that $x_1$
appears in some monomial in $f$ and its exponent in that monomial is not
divisible by $p$, so that $h:=\frac{\partial f}{\partial x_1}$ is nonzero.

If $\ul{t}$ is a $K$-point of $X_r(V)$ for which there exists a $d$-tuple
$\phi \in \Hom(V,U)$ of $K$-linear maps such that
$h(T^m_{[d]}(\phi)\ul{t})
\neq 0$, then Corollary~\ref{cor:Master} still yields the desired upper
bound on $\prk_K(\ul{t})$---indeed, we may extend $K$ to an infinite field $L$
to apply that result, and we use that the morphisms $\psi,\sigma$ from
Theorem~\ref{thm:Master} are defined over the prime field, hence over $K$.

However, what if $\ul{t}$ is a $K$-point of $X_r(V)$ such that
$h(T^m_{[d]}(\phi) \ul{t})=0$ for all $\phi \in \Hom(V,U)$? If $K$ is finite,
then this does not {\em a priori} imply that $h(T^m_{[d]}(\phi)\ul{t})=0$
holds for all $\phi \in \Hom_L(L \otimes V, L \otimes U)$, and so
$\ul{t}$
is perhaps not a point in the proper subvariety $Y' \subsetneq X_r$
defined in \eqref{eq:Yprime}.

To nevertheless obtain a bound on $\prk_K(\ul{t})$, we proceed in this case as follows. We
find a field $L$ and a $\phi \in \Hom_L(L \otimes V, L \otimes U)$
such that $h(T^m_{[d]}(\phi)\ul{t}) \neq 0$. Then Corollary~\ref{cor:Master}
gives an upper bound on the collective partition rank $\prk_L(\ul{t})$.  Finally,
we use Proposition~\ref{prop:FieldExtension} to derive an upper bound
on $\prk_K(\ul{t})$. Due to the necessity of passing to a sufficiently large field extension, the bounds we obtain for finite fields incur an additional log term.

For this to work, we need that $L$ is large enough. Observe that
$h(T^m_{[d]}(\phi)\ul{t})$ has degree $d \cdot \deg(h)$ in the entries of $\phi$.
If $|L|>d \cdot \deg(h)$, then a $\phi \in \Hom_L(L \otimes V, L \otimes U)$
exists for which this is nonzero, unless $\ul{t}$ is a $K$-point of $Y'$,
in which case we can invoke the induction hypothesis.

Hence it suffices to take for $L$ an extension of degree greater than
$ \log_2 (d \cdot \deg(h)) $---the base $2$ rather than
$|K|$ ensures that we obtain a bound that is independent on the finite
field. In what follows, $\log$ stands for $\log_2$. 

\begin{prop} \label{prop:DegBound}
For $d$ fixed and $r$ sufficiently large, $n=4(r+m^{2/d})$ satisfies
 \eqref{eq:nBound}. Choose this $n$ and set $U:=(K^n,\ldots,K^n)$. For any fixed parametrisation \eqref{eq:Parm}, let $X_{I_1,\ldots,I_r} \subset T^m_{[d]}$ be the image. Then the
minimal degree $D$ of a nonzero polynomial $f$ that vanishes identically
on $X_{I_1,\ldots,I_r}(U) $ satisfies
\[ \log(D) \ll_d \log(r+m). \]
\end{prop}

\begin{proof}
The first statement is immediate. Now consider the parameterisation in
\eqref{eq:Parm}, expanded in the standard bases of these spaces.  Along
the pull-back of this map, a coordinate on the space $T^{m}_{[d]}(U)$
is mapped to a polynomial each of whose monomials is of degree $\leq 3$
in the coordinates of $g$ and the $t_i,a_i,b_i$ (actually, these monomials
have more structure, but we ignore this). Hence a polynomial of degree
$\leq D$ on $T^{m}_{[d]}(U)$ is mapped to a polynomial each of whose
monomials is of degree $\leq 3D$ in the coordinates on the left-hand
side. The number of such monomials is at most
\begin{align*} 
\binom{m^2 + (m-1)n^d + r(n^{d-1}+n)+3D}{3D} 
\leq \binom{(m-\frac{1}{4})n^d +
3D}{(m-\frac{1}{4})n^d}=:A.
\end{align*}
We determine $D$ such that this pullback is guaranteed to have
a nontrivial kernel. For this, it suffices that there are more
degree-$\leq D$ monomials in $T^{m}_{[d]}(U)$ than the expression above.
Explicitly, it suffices that 
\[ \binom{mn^d + D}{mn^d} > A. \]
We use the well-known bounds 
\[ \left(\frac{N}{k}\right)^k \leq \binom{N}{k} \leq
\left(\frac{eN}{k}\right)^k \]
to find sufficient conditions on $D$. By the right-hand bound and
choosing $D \geq (m-\frac{1}{4})n^d$ we find 
\begin{align*} A \leq 
\left( \frac{e((m-\frac14)n^d + 3D)}{(m-\frac14)n^d}
\right)^{(m-\frac14)n^d} 
\leq 
\left( \frac{4eD}{(m-\frac14)n^d} \right)^{(m-\frac14)n^d}. 
\end{align*}
On the other hand, using the lower bound for binomial coefficients, we
find 
\[ \binom{mn^d+D}{mn^d} \geq \left( \frac{mn^d+D}{mn^d} \right)^{mn^d}
\geq \left( \frac{D}{mn^d} \right)^{mn^d}. \]
Taking logarithms, we find that it suffices that 
\begin{align*} \log(D)\cdot \frac{1}{4}n^d &\geq  mn^d \log(mn^d) 
+ 
(m-\frac14)n^d\log\left(\frac{4e}{(m-\frac14)n^d}\right). 
\end{align*}
This yields the sufficient
condition  
\[ \log(D) \geq 4\log(mn^d) 
+ 4(m-\frac14) \log\left(\frac{3e}{(m-\frac13)n^d}\right). \] 
We see that we can take $\log(D)$ a suitable constant multiple of
$\log(r+m)$, as desired. 
\end{proof}

\begin{proof}[Proof of Theorem~\ref{thm:collectiveprank} for finite
$K$]
Any $\ul{t}\in X_r(V)$ lies in $X_{I_1,\ldots,I_r}$ for some choice of proper, nonempty subsets $I_1,\ldots,I_r \subset [d].$  By Proposition~\ref{prop:DegBound}, choosing $n=4(r+m^{2/d})$ for $r$
sufficiently large and setting $U:=(K^n,\ldots,K^n)$, the minimal
degree $\delta $ of a nonzero polynomial vanishing on all of
$X_{I_1,\ldots,I_r}(U)$ satisfies $\log(\delta) \ll_d \log(r+m)$. This then also holds for the
$\Vec^d$-subvarieties $Y \subseteq X_r$ used in the proof
of Theorem~\ref{thm:collectiveprank} for infinite fields. By the
discussion preceding the proof of Proposition~\ref{prop:DegBound},
for any $K$-point $\ul{t} \in Y(V)$ we have the following dichotomy:
either $\ul{t}$ lies in the
proper subvariety $Y'(V)$ defined by \eqref{eq:Yprime}, so that the
induction hypothesis applies; or else there exists an extension $L$
of $K$ with $[L:K] \ll_d \log(r+m)$ such that $\prk_L(\ul{t})$ is at
most the polynomial $Q(r,m)$ from \eqref{eq:PolBound} (defined using $n=4(r+m^{2/d})$).
Then Proposition~\ref{prop:FieldExtension} implies that
\[ \prk_K(\ul{t}) \ll_d m^3r^{d-1}\log(r+m), \]
as desired. 
\end{proof}

\section{Proof of Theorem~\ref{thm:Df}}

We recall that Theorem~\ref{thm:Df} asks for an expression of $f$ as a polynomial in a bounded number of elements of the
space $D(f).$ 

\begin{proof}[Proof of Theorem~\ref{thm:Df}]
First we assume that $K$ is infinite and let $f$ be a $K$-point of the
variety of forms of border strength $\leq r$ in $S^d V$, where $V$ is an $n$-dimensional
vector space. Then $t:=\iota(f) \in V^{\otimes d}$ is a tensor of border partition
rank $\leq Dr$, where $D$ is as in Proposition~\ref{prop:Sym}. By Corollary~\ref{cor:Algebra}, $t$ is contained in a subalgebra generated by $\ll_d r^d$ proper contractions. This
uses the same inductive reasoning as used in the proof of Theorem~\ref{thm:collectiveprank}.

Given a proper contraction $s_j \in T_{I_j}(V),$ a
straightforward computation shows that $h_j:=\pi_{I_j}(s_j) \in S^{|I_j|}
V$ is an element of $D(f)$ and that $f = \frac{1}{d!}\pi_d(t)$ (here we use our hypothesis on $\cha(K)$)  is contained in the subalgebra generated by the $\pi_{I_j}(s_j).$ This concludes
the proof when $K$ is infinite.

When $K$ is finite, as in the proof of Theorem~\ref{thm:collectiveprank}
over finite fields, we extend $K$ to a field $L$ such that $[L:K] \ll_d \log r$ and we find that $f$ lies in an algebra
generated by $\ll_d r^d$ elements of $L \otimes_K D(f). $ As in Proposition~\ref{prop:FieldExtension}, this implies that $f$ is in an algebra generated by $\ll_d r^d\log r$ elements of $D(f).$ 
\end{proof}

\bibliographystyle{alpha}
\bibliography{diffeq,draismapreprint,draismajournal}

\end{document}